\documentclass[12pt, reqno]{amsart}
\usepackage{amsmath}
\usepackage{amssymb}
\usepackage{amsfonts}
\usepackage{mathrsfs}
\usepackage{bbding}
\usepackage{amssymb,amsmath,amscd,graphicx,latexsym}
\newtheorem{theorem}{Theorem}[section]

\newtheorem{proposition}[theorem]{Proposition}

\newtheorem{conjecture}[theorem]{Conjecture} 
\renewcommand{\a}{\alpha}
\renewcommand{\b}{\beta}

\newcommand{\f}{\frac}

\newcommand{\G}{\Gamma}
\renewcommand{\l}{\lambda}

\newcommand{\bea}{\begin{eqnarray}}
\newcommand{\eea}{\end{eqnarray}}
\newcommand{\bna}{\begin{eqnarray*}}
\newcommand{\ena}{\end{eqnarray*}}

\renewcommand{\th}{\theta}

\begin{document}

\title[The M\"obius disjointness conjecture for distal flows] 
{The M\"obius disjointness conjecture \\ for distal flows}  
\author{Jianya Liu \& Peter Sarnak}
\address{School of Mathematics
\\
Shandong University
\\
Jinan
\\
Shandong 250100
\\
China}
\email{jyliu@sdu.edu.cn}

\address{Department of Mathematics
\\
Princeton University \& Institute for Advanced Study
\\
Princeton, NJ 08544-1000       
\\
USA}
\email{sarnak@math.princeton.edu}

\date{\today}


\maketitle

\addtocounter{footnote}{1}

\section{The M\"obius  disjointness conjecture} 
\setcounter{equation}{0}
Let $\mathscr{X}=(T, X)$ be a flow, namely $X$ is a compact topological space 
and $T: X\to X$ a continuous map. The sequence $\xi(n)$ 
is observed in $\mathscr{X}$ if there is an $f\in C(X)$ and an $x\in X$,  
such that $\xi(n)=f(T^n x)$. Let $\mu(n)$ be the M\"obius function, that 
is $\mu(n)$  is $0$ if $n$ is not square-free, and is $(-1)^t$ if  $n$ is a 
product of $t$ distinct primes. We say that $\mu$ is linearly 
disjoint from $\mathscr{X}$ if 
\begin{eqnarray}\label{def/DISJO}
\f{1}{N}\sum_{n\leq N} \mu(n)\xi(n) \to 0,  \quad \mbox{as } N\to\infty, 
\end{eqnarray}
for every observable $\xi$ of $\mathscr{X}$. 
The M\"obius Disjointness Conjecture of 
the second author  (\cite{Sar}, \cite{Sar1}) states the following.  

\begin{conjecture} [The M\"obius Disjointness Conjecture]
The M\"obius function $\mu$ is linearly disjoint from 
every $\mathscr{X}$ whose entropy is $0$.  
\end{conjecture} 

This Conjecture has been established for many flows 
$\mathscr{X}$ (see \cite{Dav}, \cite{MauRiv}, \cite{GreTao2}, \cite{BouSarZie}, \cite{Bou})   
however all of these flows are quasi-regular (or rigid) in the sense that 
the Birkhoff averages 
\begin{eqnarray}\label{def/Birkh}
\f{1}{N}\sum_{n\leq N} \xi(n) 
\end{eqnarray}
exist for every  $\xi$ observed in $\mathscr{X}$. 
In \cite{LiuSar} we establish some new cases of the Disjointness Conjecture and  
in particular for irregular flows $\mathscr{X}$, that is ones for which  (\ref{def/Birkh}) 
fails.  These flows are complicated in terms of the behavior of their individual 
orbits but they are distal and of zero entropy, so that the 
disjointness is still expected to hold. 

\section{Results}  
\setcounter{equation}{0}

In this section we summarize the results we have established 
in \cite{LiuSar}. The first result in \cite{LiuSar} 
is concerned with certain regular flows, namely affine linear 
maps  of a compact abelian group $X$.  Such a flow $(T, X)$ 
is given by 
\begin{eqnarray}\label{def/AFF}
T(x)=Ax+b
\end{eqnarray}
where $A$ is an automorphism of $X$ and $b\in X$ (see \cite{Hah},  \cite{HoaPar}).  

\begin{theorem}\label{thm1} 
Let $\mathscr{X}=(T, X)$ be an affine linear flow on a compact 
abelian group which is of zero entropy. Then $\mu$ 
is linearly disjoint from $\mathscr{X}$.  
\end{theorem}

Theorem~\ref{thm1} actually holds with a rate of convergence.  
We first reduce to the torus case and then handle 
the torus case by Fourier analysis and classical results of Davenport \cite{Dav} 
and Hua \cite{Hua} on exponential sums concerning the 
M\"{o}bius function.  

The flows in Theorem~\ref{thm1} are distal, 
and our main 
result in \cite{LiuSar} is concerned with nonlinear distal flows on such spaces. 
We restrict to $X={\Bbb T}^2$ the two dimensional torus 
${\Bbb R}^2/{\Bbb Z}^2$ and consider nonlinear smooth (or even analytic) 
skew products as discussed in Furstenberg \cite{Fur61}. $T: {\Bbb T}^2\to {\Bbb T}^2$ 
is given by 
\begin{eqnarray}\label{def/SKEW}
T(x, y)= (ax+\a, cx+dy+ h(x))
\end{eqnarray}
where $a, c, d\in \Bbb Z, ad=\pm 1,  \a\in \Bbb R$ 
and $h$ is a smooth periodic function of period $1$. 
The affine 
linear part is in the form 
$$
\left[
\begin{array}{ccc}
a & 0\\
c & d
\end{array}
\right] \in GL_2 (\Bbb Z), 
$$
ensuring that $T$ has zero entropy (and it can always be brought into 
this form). The flow $(T, {\Bbb T}^2)$ is distal and this skew product 
is a basic building block (with $e(h(x))$ continuous) in Furstenberg's 
classification theory of minimal distal flows \cite{Fur63}. If 
$\a$ is diophantine, that is 
$$
\bigg|\a-\f{a}{q}\bigg|\geq \f{c}{q^B} 
$$
for some $c>0, B<\infty$ and all $a/q$ rational,  then $T$ can be conjugated 
by a smooth map of ${\Bbb T}^2$ to its affine linear part 
\begin{eqnarray}\label{DIO/conjugate}
(x, y)\mapsto (ax+\a, cx+dy+ \b)
\end{eqnarray}
where 
$$
\b=\int_0^1 h(x)dx
$$ 
(see \cite{SanUrz}). 
Hence the disjointness of $\mu$ from 
$\mathscr{X}=(T, {\Bbb T}^2)$ for a $T$ with a diophantine 
$\a$, follows from Theorem~\ref{thm1}.  
However if $\a$ is not diophantine the dynamics of the 
flow $(T, {\Bbb T}^2)$ can be very different from an affine 
linear flow. For example, as Furstenberg shows it may be 
irregular (i.e. the limits in (\ref{def/Birkh}) fail to exist  for certain observables). 
Our main result is that these nonlinear skew products 
are linearly disjoint from $\mu$, at least if $h$ satisfies some further 
technical hypothesis. Firstly we assume that $h$ is analytic, namely that 
if 
\begin{eqnarray}\label{h=/FOUR}
h(x)=\sum_{m\in \Bbb Z} \hat{h}(m) e(mx)
\end{eqnarray}
then 
\begin{eqnarray}\label{hhat/UPP}
\hat{h}(m)\ll e^{-\tau |m|}
\end{eqnarray}
for some $\tau>0$. Secondly we assume that there is 
$\tau_2<\infty$ such that 
\begin{eqnarray}\label{hhat/LOW}
|\hat{h}(m)|\gg e^{-\tau_2 |m|}.  
\end{eqnarray}
This is not a very natural condition being an artifact of our proof. 
However it is not too restrictive and the following applies rather 
generally (and most importantly there is no condition on $\a$). 

\begin{theorem}\label{thm2} 
Let $\mathscr{X}=(T, {\Bbb T}^2)$ be of the form (\ref{def/SKEW}), 
with $h$ satisfying (\ref{hhat/UPP}) and (\ref{hhat/LOW}). 
Then $\mu$ is linearly disjoint from $\mathscr{X}$.  
\end{theorem}

The assertion of 
Thereom~\ref{thm2} holds for all $\a$, and so we have to consider all 
diophantine possibilities of $\a$. The tool is the Bourgain-Sarnak-Ziegler 
\cite{BouSarZie}  
finite version of the Vinogradov method, incorporated with various analytic methods 
such as Poisson's summation and stationary phase. 
Furstenberg \cite{Fur61} gives examples of skew product transformations 
of the form (\ref{def/SKEW}) which are not regular in the sense of (\ref{def/Birkh}). 
Many of the flows ${\mathscr X}$ 
in Theorem~\ref{thm2} have this property and we show in \cite{LiuSar} 
that Furstenberg's 
examples are smoothly conjugate to such ${\mathscr X}$'s. In particular his examples are linearly 
disjoint from $\mu$.  

Theorem~\ref{thm1} deals with the affine 
linear distal flows on the $n$-torus. A different source 
of  homogeneous distal flows are the affine linear flows on nilmanifold 
$X=G/\G$ where $G$ is a nilpotent Lie group and 
$\G$ a lattice in $G$.  For 
$\mathscr{X}=(T, G/\G)$ where $T(x)=\a x\G$  
with $\a\in G$, i.e. translation on $G/\G$, the 
linear disjointness of $\mu$ and $\mathscr{X}$ is proven in 
\cite{GreTao1} and \cite{GreTao2}. Using the classification of zero entropy (equivalently distal) 
affine linear flows on nilmanifolds \cite{Dan}, and 
Green and Tao's results we establish 
in \cite{LiuSar} the following. 

\begin{theorem}\label{thm3} 
Let $\mathscr{X}=(T, G/\G)$ where $T$ is an affine 
linear map of the nilmanifold  $G/\G$ of zero entropy. 
Then $\mu$ is linearly disjoint from $\mathscr{X}$.  
\end{theorem}

The above results for $\mu(n)$ can be 
proved in the same way for similar multiplicative functions such as 
$\l(n)=(-1)^{\tau(n)}$ where $\tau(n)$ is the number of 
prime factors of $n$.

\section{Disjointness of $\mu$ from Furstenberg's systems} 
\setcounter{equation}{0}

As a consequence of Theorem~\ref{thm2}, 
it is proved in \cite{LiuSar} that $\mu$ is linearly 
disjoint from Furstenberg's systems. 
But no rate of convergence is obtained there 
since Theorem~\ref{thm2} in general offers no rate. 
In this 
section we show that, for Furstenberg's systems, rate of convergence is actually 
available if we work on these systems directly rather than appeal to  
Theorem~\ref{thm2}.    

\subsection{The continued fraction expansion of $\a$.} 
We assume that $\a$ is irrational, and our argument will 
depend on the continued fraction expansion of $\a$. 
Every real number $\a$ has its continued fraction 
representation 
\begin{eqnarray}\label{a=a0a1+}
\a=a_0+\f{1}{a_1+\f{1}{a_2+\cdots}}
\end{eqnarray}
where $a_0=[\a]$ is the integral part of $\a$, 
and $a_1, a_2, \ldots$ are positive integers. 
The expression (\ref{a=a0a1+}) is infinite since $\a\not\in \Bbb Q$.   
We write $[a_0; a_1, a_2, \ldots]$ for the 
expression on the right-hand side of (\ref{a=a0a1+}), which is the limit of 
the finite continued expressions 
\begin{eqnarray}\label{a=a0a1F}
[a_0; a_1, a_2, \ldots, a_k]=a_0+\f{1}{a_1+\f{1}{a_2+\cdots+\f{1}{a_k}}}
\end{eqnarray}
as $k\to \infty$. 
Writing 
\begin{eqnarray*}
\f{l_k}{q_k}=[a_0; a_1, a_2, \ldots, a_k], 
\end{eqnarray*}
we have $l_0=a_0, l_1=a_0a_1+1, q_0=1, q_1=a_1, $ 
and for $k\geq 2$, 
\begin{eqnarray*}
l_k=a_k l_{k-1}+l_{k-2}, \quad 
q_k=a_k q_{k-1}+q_{k-2}.  
\end{eqnarray*}
Since $\a$ is irrational we have $q_{k+1}\geq q_k+1$ for all $k\geq 1$. 
An induction argument gives the stronger 
assertion that 
$q_k\geq 2^{(k-1)/2}$ 
for all $k\geq 2$, and thus $q_k$ increases at least like an exponential function of $k$.  
The irrationality of $\a$ also implies that, for all $k\geq 2$, 
\begin{eqnarray}\label{ratAPP}
\f{1}{2 q_k q_{k+1}}<\bigg|\a-\f{l_k}{q_k}\bigg| <\f{1}{q_kq_{k+1}}. 
\end{eqnarray} 

\subsection{Furstenberg's examples.}  
Furstenberg \cite{Fur61} 
gave examples of smooth transformation $T: {\Bbb T}^2\to {\Bbb T}^2$ 
such that the ergodic averages do not all exist. 
Let $\a$ be as above such that 
\begin{eqnarray}\label{qk+1>eqk}
q_{k+1}\geq e^{q_k} 
\end{eqnarray} 
for all positive $k$.  
Define $q_{-k}= - q_k$ and set   
\begin{eqnarray}
h(x)=\sum_{k\not=0}\f{e(q_k\a)-1}{|k|} e(q_k x). 
\end{eqnarray}
It follows from (\ref{ratAPP}) and (\ref{qk+1>eqk}) that 
$h(x)$ is a smooth function. We also have 
$h(x)=g(x+\a)-g(x)$ 
where 
\begin{eqnarray}
g(x)=\sum_{k\not=0}\f{1}{|k|}e(q_k x)
\end{eqnarray}
so that $g(x)\in L^2(0,1)$ and in particular   
defines and measurable function. But $g(x)$ cannot 
correspond to a continuous function, as shown in Furstenberg \cite{Fur61}.

\subsection{Disjointness of $\mu$ from Furstenberg's systems.}    
In the following we consider slightly more general $h$'s with  
\begin{eqnarray}\label{def/hx=sum}
h(x)=\sum_{k\not=0} c_k (1-e(q_k\a))e(q_k x)
\end{eqnarray}
where $\a$ satisfies (\ref{qk+1>eqk}) and, for all positive $k$, the coefficients 
$c_k$ satisfy  
\begin{eqnarray}
c_k= c_{-k}, \quad |c_k|\leq C 
\end{eqnarray} 
for some positive constant $C$ which we may 
assume to be greater than $1$. 
We note that if $c_k$ is not bounded then $g(x)$ will not be $L^2$. 
Now what we want to estimate is essentially   
\begin{eqnarray}\label{defSN}
S(N):=\sum_{n\leq N} \mu(n) e\bigg(\sum_{j=0}^{n-1}h(j\a)\bigg), 
\end{eqnarray}
and our result is the following. 

\begin{proposition}\label{thm:2-2}
Let $S(N)$ be as in (\ref{defSN}). Then
\begin{eqnarray}\label{S2est}
S(N)\ll_{} N\log^{-A}N
\end{eqnarray}
where the implied constant depends on $A$, 
but is independent of $\a$.  
\end{proposition} 

\begin{proof} 
By (\ref{def/hx=sum}) we have 
\begin{eqnarray*}
\sum_{j=0}^{n-1}h(j\a)
=\sum_{k\not=0} c_k (1-e(q_k\a))\sum_{j=0}^{n-1}e(q_k j\a) 
=\sum_{k\not=0} c_k (1-e(n q_k\a)).  
\end{eqnarray*}
We should cut the last sum at some point. Let $K$ be such that 
$q_{K-1}< 2\log N \leq q_{K}$. 
Then for $k\geq K$ we have 
\begin{eqnarray}\label{qk-lk}
|q_k\a-l_k|
\leq \f{1}{q_{k+1}} 
\leq e^{-q_k}, 
\end{eqnarray} 
so that 
\begin{eqnarray*}
\sum_{|k|\geq K} |c_k (1-e(n q_k\a))| 
\ll C\sum_{|k|\geq K} n e^{-q_{k}}  
\ll  CN e^{-q_{K}}  
\ll  CN^{-1}  
\end{eqnarray*} 
where we the implied constant does not depend on $C$. 
It follows that 
\begin{eqnarray*}
S(N)=\sum_{n\leq N} \mu(n) 
e\bigg\{\sum_{1\leq |k|\leq K-1} c_k (1-e(n q_k\a))+O\bigg(\f{C}{N}\bigg)\bigg\}. 
\end{eqnarray*}
The $O$-term as well as $\sum\limits_{1\leq |k|\leq K-1} c_k $ 
above are harmless, 
and so from now on we concentrate on 
\begin{eqnarray}\label{TilS=}
\widetilde{S}(N) 
&:=&\sum_{n\leq N} \mu(n) 
e\bigg\{\sum_{1\leq |k|\leq K-1} c_k e(n q_k\a)\bigg\} \nonumber\\ 
&=&\sum_{n\leq N} \mu(n) 
e\bigg\{\sum_{1\leq |k|\leq K-1} c_k e(n \th_k)\bigg\} 
\end{eqnarray}
on writing  $\th_k=\|q_k\a\|$. Note that by 
(\ref{qk-lk}) we have $\th_k\leq e^{-q_k}$.

We have a sequence 
$$
q_1<\exp(q_1) \leq q_2 < \exp(q_2) \leq q_3< \ldots,    
$$
and therefore  
$q_k> \exp \cdots\exp (q_2)$ 
with $k-2$ repeated $\exp$'s.  
Taking $k=K-1$ and noting that $q_2\geq 2$ 
give  
\begin{eqnarray}\label{repEXP}
\exp \cdots\exp (2) \leq \exp \cdots\exp (q_2) < q_{K-1}\leq 2\log N
\end{eqnarray}
where on the left-hand side there are $K-3$ $\exp$'s.  
 
Now let 
\begin{eqnarray}\label{DEFfi}
\phi(x)
= 2c_1\cos (2\pi x), \quad x\in [\th_1,\th_1 N]. 
\end{eqnarray}
Then $e(\phi(x))$ is a smooth periodic function  
and hence can expanded into 
Fourier series 
\bea 
e(\phi(x))=\sum_{l\in \Bbb Z} a_l(c_1)e(lx), 
\eea 
where 
\bea 
a_l(c_1)=\int_0^1 e(\phi(x))e(-lx)dx. 
\eea 
We must compute the dependence of $a_l(c_1)$ on $c_1$. 
By partial integration for $l\not=0$ we have 
\bna 
a_l(c_1)
&=& -\f{1}{2\pi i l}\int_0^1 e(\phi(x))d e(-lx) \\ 
&=& -\f{1}{l}\int_0^1 e(\phi(x))\phi'(x) e(-lx) dx\\ 
&=& \f{1}{2\pi i l^2}\int_0^1 \f{d[e(\phi(x))\phi'(x)]}{dx} e(-lx) dx. 
\ena 
Since 
\bna 
\phi'(x)= - 4\pi c_1\sin (2\pi x), \quad 
\phi''(x)= - 8\pi^2 c_1\cos (2\pi x), 
\ena 
we have 
\bna 
\bigg|\f{d[e(\phi(x))\phi'(x)]}{dx}\bigg|
=|e(\phi(x)) [\phi''(x)+2\pi i\phi'(x)\phi'(x)]|
\ll |c_1|+|c_1|^2\ll C^2
\ena 
and consequently 
\bea\label{al<<} 
a_l(c_1)
\ll \f{C^2}{l^2} 
\eea 
for all $l\not=0$. 
The above $\ll$-constant is absolute. Of course the above $l^2$ 
can be improved to $l^A$ for arbitrary $A>0$, but we are not going to 
use this. 

The sum $\widetilde{S}(N)$ in (\ref{TilS=}) is of the form 
\begin{eqnarray*}
\widetilde{S}(N) 
&=&\sum_{n\leq N} \mu(n)e(\phi(n\th_1)+F(n))\\ 
&=& \sum_{n\leq N} \mu(n) e(F(n))\sum_{l\in \Bbb Z} a_{l}(c_1) e(ln\th_1)\\ 
&=& \sum_{l\in \Bbb Z} a_{l}(c_1) 
\sum_{n\leq N} \mu(n) e(ln\th_1+F(n)), 
\end{eqnarray*} 
where $F(n)$ stands for the remaining terms in the $e(\cdot)$ in (\ref{TilS=}).  
It follows from this and (\ref{al<<}) that 
\begin{eqnarray*}
|\widetilde{S}(N)|  
&\leq & \sum_{l\in \Bbb Z} |a_{l}(c_1) | 
\bigg|\sum_{n\leq N} \mu(n) e(ln\th_1+F(n))\bigg|\\ 
&\ll & C^2 \sum_{l\in \Bbb Z}\f{1}{l^2+1}
\bigg|\sum_{n\leq N} \mu(n) e(ln\th_1+F(n))\bigg|\\ 
&\ll & C^2   
\sup_{l_1} \bigg|\sum_{n\leq N} \mu(n) e(n l_1\th_1 +F(n))\bigg|. 
\end{eqnarray*}
Repeating this procedure in (\ref{TilS=}), we get 
\begin{eqnarray*}
|\widetilde{S}(N)| 
\ll  C^{2(K-1)}  
\sup_{l_1, \ldots, l_{K-1}} \bigg|\sum_{n\leq N} 
\mu(n) e(n l_1\th_1+ \ldots+ n l_{K-1}\th_{K-1}) \bigg|.  
\end{eqnarray*}
The inner sum involving $\mu$ can be estimated by classical results of 
Davenport \cite{Dav} and Hua \cite{Hua} as 
$$
\ll N\log^{-A}N
$$ 
where $A>0$ is arbitrary, and the implied constant 
depends at most on $A$, i.e. it does not depend on the coefficients 
$l_1, l_2, \ldots, l_{K-1}$ or $\th_1, \th_2, \ldots, \th_{K-1}$. 
By (\ref{repEXP}) we have $C^{2(K-1)}\leq \log N$. The proposition is proved.   
\end{proof} 

\medskip 
\noindent 
{\bf Acknowledgements.} 
The first author is supported by the 973 Program, 
NSFC grant 11031004, and IRT1264 from the Ministry of Education. 
The second author is supported by an NSF grant.

\end{document}